\documentclass{scrartcl}

\usepackage{natbib}
\usepackage{geometry}
\usepackage{fleqn}
\usepackage{graphicx}
\usepackage[colorlinks,citecolor=blue,urlcolor=blue]{hyperref}
\usepackage{xspace,enumerate} 
\usepackage{amsmath,amssymb}
\usepackage{enumitem}
\usepackage{cite}
\usepackage{mathtools}
\usepackage{dsfont,mathrsfs}
\usepackage{amsthm}
\usepackage{bbm}
\usepackage{tabularx}
\usepackage{booktabs}

\newcommand{\cD}{\mathcal{D}}
\newcommand{\cG}{\mathcal{G}}
\newcommand{\cF}{\mathcal{F}}
\newcommand{\cP}{\mathcal{P}}

\newcommand{\NN}{\mathbb{N}}
\newcommand{\RR}{\mathbb{R}}
\newcommand{\ind}{\mathbbm{1}}

\numberwithin{equation}{section}

\theoremstyle{plain}
\newtheorem{theorem}{Theorem}[section]
\numberwithin{theorem}{section}
\newtheorem{corollary}[theorem]{Corollary}

\newtheorem{proposition}[theorem]{Proposition}

\newtheorem{example}[theorem]{Example}
\newtheorem{remark}[theorem]{Remark}

\DeclareMathOperator{\var}{Var}
\DeclareMathOperator{\supp}{supp}
\DeclareMathOperator{\id}{id}

\title{Comparison of path-independent functions of semimartingales}
\author{Benedikt K\"opfer\footnote{A LGFG grant of the state Baden-W\"urttemberg is gratefully acknowledged}, Ludger R\"uschendorf}
\date{}

\begin{document}

\maketitle
\thispagestyle{empty}

\begin{abstract}
The martingale comparison method is extended to derive comparison results for path-independent functions for general semimartingales. Our approach allows to dismiss with the Markovian assumption on one of the processes made in previous literature. Main ingredients of the comparison method are extensions of the Kolmogorov backwards equation to the non-Markovian case. Putting the comparison processes on the same stochastic basis allows by means of It\^o's formula applied to the propagation operator to conclude the comparison of the processes from the comparison of the semimartingale characteristics.
\end{abstract}

\renewcommand{\thefootnote}{}
\footnotetext{\hspace*{-.51cm}%
AMS 2010 subject classification:
Primary: 60E15; secondary: 60G44, 60G51.\\
Key words and phrases: Path-independent ordering; ordering of semimartingales; Kolmogorov backwards equation}

{\renewcommand{\thefootnote}{\arabic{footnote}}

\section{Introduction}
\label{sec:intro}

Mainly motivated by the problem of deriving ordering results for option prices, comparison results have been derived in \citet{KJS98}, \citet{Ho98}, \citet{BJ00} and \citet{He05}. \citet{GM02} developed a general approach to comparison results w.r.t. to convex ordering of terminal values between one-dimensional semimartingales and Markovian semimartingales based on the supermartingale property of a linking process - the martingale comparison method. Essentially the comparison of local (differential) semimartigale characteristics and the 'propagation of convexity' property of the Markov process imply convex ordering under the assumption that the propagation operator (the value process) of the Markov process satisfies a Kolmogorov backwards equation. Some extensions of this martingale comparison method are given in \citet{BR06,BR07a,BR07b,BR08}. In particular in these papers a general version of the Kolmogorov backwards equation for Markov processes is establihed and extensions to multivariate processes, to further orderings and to some classes of path-dependent options are given.

In the present paper this approach is generalized allowing to state comparison results between two general semimartingales. In comparison to the papers mentioned above, we use the same stochastic basis $\left (\Omega,\cF,(\cF_t)_{[0,T]},P\right )$ for both semimartingales under consideration. This has the advantage that the semimartingale characteristics can be chosen more freely. In the papers above a standing assumption is that one of the processes is a Markov process such that the differential characteristics are functions of the space-time process. This is not necessary if the semimartingales are on the same stochastic basis and we are able to compare two semimartingales directly. Additionally we do not restrict the characteristics to be absolutely continuous. So the results can be applied for example to semimartingales with fixed jump times.

In Section \ref{sec:funceq} we specify the setting and notation. The basic tool in our paper for the proof of the comparison theorems is an extension of the Kolmogorov backwards equation for Markov processes in \citet{BR06,BR07a} to the case of special semimartingales. We give a formulation of these extensions as ``functional equations'' allowing in principle also applications different from the case of backward equations for the pricing functional.

In Section \ref{sec:compemm} we derive comparison results under equivalent martingale measures (e.m.m.). For two semimartingales $X$ and $Y$ with corresponding e.m.m. $Q_1$ and $Q_2$, we state for an integrable function $f$ conditions such that 
\begin{align*}
E_{Q_2}[f(Y_T)] \le E_{Q_1}[f(X_T)].
\end{align*}
The main tool therein is the factorized conditional expectation (propagation operator) of $X$
\begin{align*}
G_f(t,x) := E_{Q_1}[f(X_T)|X_t = x].
\end{align*}
For $G_f$ in $C^{1,2}$ we consider the basic linking process $G_f(t,Y_t)$ which allows by It\^{o}'s formula to link the semimartingale characteristics of $X$ and $Y$. In the subsequent Section \ref{sec:compsm} we derive similar comparison results for two special semimartingales under $P$.

In Section \ref{sec:discussion}, we discuss the assumptions of the main comparison theorems and give examples of classes of semimartingales which possess the required regularity properties. In particular, we discuss the assumption that $G_f$ is of class $C^{1,2}$ and that $G_f$ is convex or directionally convex in the second variable. We conclude this paper with some examples.

\section{Functional equations for local martingales}
\label{sec:funceq}

We consider a finite time horizon since we are interested in the comparison of the processes at fixed time points; so we can take this point as final time point. Let $(X_t)_{t \in [0,T]}$ be an $\RR^d$ valued special semimartingale on a stochastic basis $(\Omega, \cF, (\cF_t)_{[0,T]}, P)$. Further, let $X = M + B$ be the canonical decomposition of $X$ into a local martingale $(M_t)_{t \in [0,T]}$ and a process of finite variation $(B_t)_{t \in [0, T]}$. The $d+1$-dimensional space-time Process $\hat{X} := ((t,X_t))_{t \in [0,T]}$ then also is a special semimimartingale; thus we can choose the truncation function for the semimartingale characteristics to be the identity, even though it is not a truncation function in the sense of \citet{JS03}, see \citet{RS11}. The local martingale part of the canonical decomposition is $(0,M)$ and the finite variation part is $(\id,B)$. By \citet[Proposition II.2.9]{JS03} there exists a predictable process $\hat{A} := (\hat{A}_t)_{t \in [0,T]} \in \mathscr{A}_{loc}^+$ such that the semimartingale characteristics of $\hat{X}$ are given as Lebesgue--Stieltjes integrals with respect to $\hat{A}$. The characteristics $(\hat{B}, \hat{C}, \hat{\nu})$ have the ``good'' form:
\begin{align*}
\begin{cases}
\hat{B}^i = \hat{b}^i \cdot \hat{A},\\
\hat{C}^{ij} = \hat{c}^{ij} \cdot \hat{A},\\
\hat{\nu}(\omega, dt, dx) = d\hat{A}_t(\omega) \hat{K}_{\omega, t}(dx),
\end{cases}
\end{align*}
When there is no danger of confusion, we only write differential characteristics without specifying the integrator process. One candidate process $\hat{A}$ is explicitly specified, namely
\begin{align}
\label{eq:Afromniceversion}
\hat{A} = \sum_{i \le d} \var(\hat{B}^i) + \sum_{i,j \le d} \var(\hat{C}^{ij}) + (|x|^2 \wedge 1) \ast \hat{\nu}.
\end{align} 
Altogether we obtain the following version of the canonical decomposition of the space-time process $\hat{X}$:
\begin{align*}
\hat{X}_t = (0,X_0) + (0,M_t) + (t,B_t)= \hat{X}_0 + (0, M_t) + ((\hat{b} \cdot \hat{A})_t).
\end{align*}
The integral in the last term is understood componentwise. Note that the change from $X$ to $\hat{X}$ does not change the semimartingale characteristics of $X$, they are still contained in the last $d$ dimensions. Only the differential characteristics change because we look for a common integrator. In the theory of Markov processes it is a common procedure to consider the space-time process. Here the space-time process helps to connect the time derivative and the space derivatives.

We start with equations which characterise $C^{1,2}$ functions of $\hat{X}$ which are local martingales. Since we use It\^{o}'s formula, we introduce for a function $f \in C^{0,1}(\RR_+ \times \RR^d)$ the following function:
\begin{align}
\begin{split}
\label{eq:Hforjumpintegral}
H_f:~ \RR_+ \times \RR^d \times \RR^d \to&~ \RR,\\
(t,x,y)~~~~ \mapsto&~ f(t,x+y) - f(t,x) - \sum_{i \le d} \frac{\partial}{\partial x^i} f(t,x)y^i.
\end{split}
\end{align}
We write $dA$ for the measure associated to a process of finite variation $A$. In the sequel we use the following notation for function classes:
\begin{align*}
&\cF_i := \{ f: \RR^d \to \RR ; f ~\text{is increasing} \}, \cF_{dcx} := \{ f: \RR^d \to \RR ; f ~\text{is directionally convex} \},\\
&\cF_{cx} := \{ f: \RR^d \to \RR ; f ~\text{is convex} \}, \cF_{icx} := \{ f: \RR^d \to \RR ; f ~\text{is increasing and convex} \},\\ 
&\cF_{idcx} := \{ f: \RR^d \to \RR ; f ~\text{is increasing and directionally convex} \}.
\end{align*} 

In this section we assume first that the semimartingale $X$ under consideration is a local martingale. Then the canonical decomposition of $\hat{X}$ reduces to 
\begin{align*}
\hat{X} = (0,X_0) + (0,X) + (\hat{b} \cdot \hat{A}) = (0,X_0) + (0,X) + (\id,0).
\end{align*}
Note that here particularly $(\hat{b} \cdot \hat{A}) = (\id,0)$.

The following lemma is an extension of Kolmogorv's backward equation for Markov processes in \citet{BR06} to local martingales.

\begin{proposition}
\label{lemma:kolmogorovbackwardemm}
Let $f \in C^{1,2}([0,T] \times \RR^d)$ and let $X$ be a local martingale. Assume that:
\begin{itemize}
\item[(i)] $(f(t,X_t))_{t \ge 0}$ is a local martingale;

\item[(ii)] $\big|H_f \big| \ast \mu^X \in \mathscr{A}_{loc}^+$.
\end{itemize}

Then the following process is $d\hat{A} \times P$ almost surely identical zero
\begin{align}
\begin{split}
\label{eq:kolmogorovbackwardemm}
U_tf(t,X_{t^-}) :=&~ \hat{b}_t \frac{\partial}{\partial t} f(t,X_{t^-}) + \frac{1}{2} \sum_{i,j \le d} \hat{c}^{ij}_t \frac{\partial^2}{\partial x^i x^j} f(t,X_{t^-})\\
&~ + \int_{\RR^d} H_f(t,X_{t^-},x) \hat{K}_t(dx) = 0.
\end{split}
\end{align}
\end{proposition}

\begin{proof}
By It\^{o}'s formula the local martingale $(f(t,X_t))_{t \ge 0}$ has the following representation
\begin{align*}
f(t, X_t) = &~ f(0,X_0) + \int_0^t \frac{\partial}{\partial s} f(s,X_{s^-}) \hat{b}_s d\hat{A}_s + \sum_{i \le d} \int_0^t \frac{\partial}{\partial x^i} f(s, X_{s^-}) dX_s^i\\
&~ + \frac{1}{2} \sum_{i,j \le d} \int_0^t \frac{\partial^2}{\partial x^i x^j} f(s,X_{s^-}) \hat{c}^{ij}_s d\hat{A}_s\\ 
&~ + \int_{[0,t] \times \RR^d} \left[ f(s, X_{s^-} + x) - f(s, X_{s^-}) - \sum_{i \le d} \frac{\partial}{\partial x^i} f(s, X_{s^-}) x^i \right] \mu^X (ds,dx).
\end{align*}
We compensate the jump integral, which is possible by Assumption~$(ii)$ and \citet[Proposition II.1.28]{JS03}. Denoting
{\small
\begin{align*}
& M_t := \sum_{i \le d} \int_0^t \frac{\partial}{\partial x^i} f(s, X_{s^-}) dX_s^i,\\
& N_t :=\int_{[0,t] \times \RR^d} \left [ f(s, X_{s^-} + x) - f(s, X_{s^-}) - \sum_{i \le d} \frac{\partial}{\partial x^i} f(s, X_{s^-}) x^i \right ] [\mu^X (ds,dx) - \hat{K}_s(dx) d\hat{A}_s],
\end{align*}}
the processes $(M_t)_{t \in [0,T]}$ and $(N_t)_{t \in [0,t]}$ are local martingales. As consequence we obtain
\begin{align*}
f(t, X_t) = &~ f(0,X_0) + \int_0^t \hat{b}_s \frac{\partial}{\partial s} f(s,X_{s^-}) d\hat{A}_s + M_t + N_t\\
&~ + \frac{1}{2} \sum_{i,j \le d} \int_0^t \hat{c}^{ij}_s \frac{\partial^2}{\partial x^i x^j} f(s,X_{s^-}) d\hat{A}_s + \int_{[0,t] \times \RR^d} H_f(s,X_{t^-},x) \hat{K}_s(dx)d\hat{A}_s.
\end{align*}
It follows that the process
\begin{align*}
&~\int_0^t \left[ \hat{b}_s \frac{\partial}{\partial s} f(s,X_{s^-}) + \frac{1}{2} \sum_{i,j \le d} \hat{c}^{ij}_s \frac{\partial^2}{\partial x^i x^j} f(s,X_{s^-}) + \int_{\RR^d} H_f(s,X_{t^-},x) \hat{K}_s(dx) \right] d\hat{A}_s
\end{align*}
is a predictable local martingale of finite variation starting in zero and is, therefore, almost surely zero by \citet[Corollary I.3.16]{JS03}. Thus, the integrand has to be $d\hat{A} \times P$ almost surely zero as well. 
\end{proof}

We remark that $f(t, \cdot) \in \mathscr{F}_{cx}$ implies condition (ii) (see \citet{BR06}).

Next we obtain a similar equation in the case that $X$ is a special semimimartingale. Note that the process $\hat{X}$ then is a special semimartingale as well and we have as truncation function the identity. The canonical decomposition of $\hat{X}$ has the form:
\begin{align*}
\hat{X}_t = \hat{X}_0 + \left(0,X_t^c + x \ast (\mu^X - \nu)_t \right) + (\hat{b}^X \cdot \hat{A})_t.
\end{align*}
The reason why we demand $X$ to be special is that we then are able to compensate all of the jumps appearing in It\^{o}'s formula directly. For a general semimartingale the canonical decomposition with a truncation function $h$ is 
\begin{align*}
\hat{X}_t = \hat{X}_0 + \left(0,X_t^c + h \ast (\mu^X - \nu)_t \right) + \left (0,(x-h(x))\ast \mu^X\right )_t + (\hat{b}^X \cdot \hat{A})_t.
\end{align*}
Hence, an integral with respect to $\mu^X$ is added in It\^o's formula. This makes an additional assumption necessary. However, this turns the $H_f$ term into a term with a truncation function which leads to analog proofs; we omit details here.

The following proposition is a version of the Kolmogorov backward equation for special semimartingales.

\begin{proposition}
\label{lemma:kolmogorovbackwardp}
Let $f \in C^{1,2}([0,T] \times \RR^d)$ and let $X$ be a special semimartingale. Assume that:
\begin{itemize}
\item[(i)] $(f(t,X_t))_{t \ge 0}$ is a local martingale;

\item[(ii)] $\big|H_f \big| \ast \mu^X \in \mathscr{A}_{loc}^+$.
\end{itemize}
Then the following process is $d\hat{A} \times P$ almost surely zero
\begin{align}
\begin{split}
\label{eq:kolmogorovbackwardp}
\bar{U}_tf(t,X_{t^-}) :=&~ \hat{b}_t \frac{\partial}{\partial t} f(t,X_{t^-}) + \sum_{i \le d} \hat{b}^i_t \frac{\partial}{\partial x^i} f(t, X_{t^-}) + \frac{1}{2} \sum_{i,j \le d} \hat{c}^{ij}_t \frac{\partial^2}{\partial x^i x^j} f(t,X_{t^-})\\
&~ + \int_{\RR^d} H_f(t,X_{t^-},x) \hat{K}_t(dx) = 0.
\end{split}
\end{align}
\end{proposition}

\begin{proof}
The proof is similar to the proof of Lemma \ref{lemma:kolmogorovbackwardemm}. It\^{o}'s formula yields
\begin{align*}
f(t, X_t) = &~ f(0,X_0) + \int_0^t \frac{\partial}{\partial s} f(s,X_{s^-}) \hat{b}_s d\hat{A}_s + \sum_{i \le d} \int_0^t \frac{\partial}{\partial x^i} f(s, X_{s^-}) dX_s^i\\
&~ + \frac{1}{2} \sum_{i,j \le d} \int_0^t \frac{\partial^2}{\partial x^i x^j} f(s,X_{s^-}) \hat{c}^{ij}_s d\hat{A}_s\\ 
&~ + \int_{[0,t] \times \RR^d} \left[ f(s, X_{s^-} + x) - f(s, X_{s^-}) - \sum_{i \le d} \frac{\partial}{\partial x^i} f(s, X_{s^-}) x^i \right] \mu^X (ds,dx).
\end{align*}
We compensate the jumps (by Assumption~$(ii)$) and use the canonical decomposition of $X$ to split the $dX$ term into the local martingale part and the part of finite variation. We obtain with $M_t = X_t^c + x \ast (\mu^X - \hat{\nu})_t$
\begin{align*}
f(t, X_t) = &~ f(0,X_0) + \int_0^t \frac{\partial}{\partial s} f(s,X_{s^-}) \hat{b}_s d\hat{A}_s + \sum_{i \le d} \int_0^t \frac{\partial}{\partial x^i} f(s, X_{s^-}) \hat{b}^i_s d\hat{A}_s\\
&~ + \sum_{i \le d} \int_0^t \frac{\partial}{\partial x^i} f(s, X_{s^-}) dM_s^i + \frac{1}{2} \sum_{i,j \le d} \int_0^t \frac{\partial^2}{\partial x^i x^j} f(s,X_{s^-}) \hat{c}^{ij}_s d\hat{A}_s\\ 
&~ + \int_{[0,t] \times \RR^d} H_f(s,X_{s^-},x) \left [\mu^X (ds,dx) - \hat{K}_s(dx) d\hat{A}_s\right ]\\
&~ + \int_{[0,t] \times \RR^d} H_f(s,X_{s^-},x) \hat{K}_s(dx) d\hat{A}_s
\end{align*}
We conclude that
\begin{align*}
&~\int_0^t \left[ \frac{\partial}{\partial s} f(s,X_{s^-}) \hat{b}_s + \sum_{i \le d} \frac{\partial}{\partial x^i} f(s, X_{s^-}) \hat{b}^i_s + \frac{1}{2} \sum_{i,j \le d} \frac{\partial^2}{\partial x^i x^j} f(s,X_{s^-}) \hat{c}^{ij}_s \right.\\
&~ \left. + \int_{\RR^d} H_f(s,X_{s^-},x) \hat{K}_s(dx) \vphantom{\sum_{i,j \le d} \frac{\partial^2}{\partial x^i x^j}} \right] d\hat{A}_s
\end{align*}
is a predictable local martingale of finite variation starting in zero and is therefore almost surely zero. Hence, the integrand has to be $d\hat{A} \times P$ almost surely zero as well.
\end{proof}

\section{Comparison under equivalent martingale measures}
\label{sec:compemm}

In this section we establish comparison results for semimartingales by the martingale comparison method. Let $X = (X_t)_{t \in [0,T]}$ and $Y = (Y_t)_{t \in [0,T]}$ be semimartingales. Assume that there exist equivalent martingale measures $Q_1$ and $Q_2$ on $(\Omega,\cF,(\cF_t)_{t \in [0,T]})$ for $X$ and $Y$ each, i.e. $X$ is local martingale under $Q_1$ and $Y$ is a local martingale under $Q_2$. In the sequel we denote semimartingale characteristics with a superscript such that it is clear to which process they belong. Further, denote by $\hat{X} := (\id,X)$ and $\hat{Y} := (\id,Y)$ the corresponding space-time processes. For an equivalent local martingale measure $Q_1$ for $X$ and a measurable function $f: (\RR^d,\mathscr{B}(\RR^d)) \to (\RR,\mathscr{B}(\RR))$ such that $f(X_T) \in L^1(Q_1)$ we introduce the pricing functional $G$ (propagation operator)
\begin{align}
\label{eq:backwardfunctional}
G_f(t, x) := E_{Q_1}[f(X_T)|X_t = x].
\end{align}
Note that in the sequel the semimartingale characteristics of $\hat{X}$ are w.r.t. $Q_1$, whereas the characteristics of $\hat{Y}$ are w.r.t. $Q_2$. The semimartingale characteristics under the particular e.m.m. can be obtained by the Girsanov theorem, see \citet[Theorem III.3.24]{JS03}.

The following directionally convex comparison theorem is an extension of \citet[Theorem 2.3]{BR06} to non-Markovian semimartingales.

\begin{theorem}[Directionally convex comparison under e.m.m.] 
\label{thm:orderingdcxemm}
Let $X, Y$ be semimartingales and let $X_0 = Y_0 = x_0 \in \RR^d$ almost surely. We consider a function $f$ such that $f(X_T) \in L^1 (Q_1)$ and $f(Y_T) \in L^1 (Q_2)$. Assume that
\begin{itemize}
\item[(i)] $G_f \in C^{1,2}([0,T] \times \RR^d)$ and $G_f(t, \cdot) \in \cF_{dcx}$ for all $t \in [0,T]$;

\item[(ii)] $U^X_t G_f(t,Y_{t^-}) = 0$ holds $dA^{\hat{Y}} \times Q_2$ almost surely for all $t \in [0,T]$, where $U^X_t$ is the operator defined in \eqref{eq:kolmogorovbackwardemm} with the differential semimartingale characteristics of $\hat{X}$ under $Q_1$ in it;

\item[(iii)] $\big|H_{G_f} \big| \ast \mu^Y \in \mathscr{A}_{loc}^+$, where $H_{G_f}$ is defined in \eqref{eq:Hforjumpintegral};

\item[(iv)] $(G_f(t, Y_t)^-)_{t \in [0,T]}$ is of class (DL);

\item[(v)] $A^{\hat{Y}} = A^{\hat{X}}$;

\item[(vi)] The differential characteristics are $dA^{\hat{Y}} \times Q_2$ almost surely ordered, for all $i,j \le d$:
\begin{align*}
c_t^{\hat{Y} ij} \le &~ c_t^{\hat{X} ij},\\
\int_{\RR^d} g(t,Y_{t^-},x) K^{\hat{Y}}_t(dx) \le&~ \int_{\RR^d} g(t,Y_{t^-},x) K^{\hat{X}}_t(dx), 
\end{align*}
where the second inequality holds for all $g(t,y,\cdot) \in \cF_{dcx}$ such that the integrals exist.
\end{itemize} 

Then it holds that
\begin{align*}
E_{Q_2}[f(Y_T)] \le E_{Q_1}[f(X_T)].
\end{align*}
If in $(vi)$ the inequalities are reversed and $(G_f(t, Y_t)^+)_{t \in [0,T]}$ is of class (DL), then we have that 
\begin{align*}
E_{Q_2}[f(Y_T)] \ge E_{Q_1}[f(X_T)].
\end{align*}
\end{theorem}

\begin{proof}
We establish that the linking process $(G_f(t, Y_t))_{t \in [0,T]}$ is a $Q_2$-supermartingale. This is the key idea of the martingale comparison method. Then the assertion follows from the inequality
\begin{align*}
E_{Q_2}[f(Y_T)] = E_{Q_2}[G_f(T,Y_T)] \le G_f(0,x_0) = E_{Q_1}[f(X_T)].
\end{align*}
Since $G_f \in C^{1,2}([0,T] \times \RR^d)$, It\^{o}'s formula yields that $(G_f(t, Y_t))_{t \in [0,T]}$ is a semimartingale starting in $G_f(0,x_0)$ with decomposition
{\small
\begin{align*}
G_f(t, Y_t) =&~ G_f(0,x_0) + \int_0^t \frac{\partial}{\partial s} G_f(s,Y_{s^-}) b^{\hat{Y}}_s dA^{\hat{Y}}_s + \sum_{i \le d} \int_0^t \frac{\partial}{\partial x^i} G_f(s, Y_{s^-}) dY_s^i\\
&~ + \frac{1}{2} \sum_{i,j \le d} \int_0^t \frac{\partial^2}{\partial x^i x^j} G_f(s,Y_{s^-}) c^{\hat{Y}ij}_s dA^{\hat{Y}}_s\\ 
&~ + \int_{[0,t] \times \RR^d} \left[ G_f(s, Y_{s^-} + x) - G_f(s, Y_{s^-}) - \sum_{i \le d} \frac{\partial}{\partial x^i} G_f(s, Y_{s^-}) x^i \right] \mu^Y (ds,dx).
\end{align*}}
We compensate the jumps (which is possible because of Assumption~$(iii)$) and define
\begin{align*}
M_t := \sum_{i \le d} \int_0^t \frac{\partial}{\partial x^i} G_f(s, Y_{s^-}) dY_s^i + \int_{[0,t] \times \RR^d} H_{G_f}(s,X_{s^-},x) \left [\mu^Y (ds,dx) - K^{\hat{Y}}_s(dx) dA^{\hat{Y}}_s\right ].
\end{align*}
Then we obtain
\begin{align*}
G_f(t, Y_t) = &~ G_f(0,x_0) + \int_0^t \frac{\partial}{\partial s} G_f(s,Y_{s^-}) b^{\hat{Y}}_s dA^{\hat{Y}}_s + M_t \\
&~ + \frac{1}{2} \sum_{i,j \le d} \int_0^t \frac{\partial^2}{\partial x^i x^j} G_f(s,Y_{s^-}) c^{\hat{Y}ij}_s dA^{\hat{Y}}_s + \int_{[0,t] \times \RR^d} H_{G_f}(s,Y_{s^-},x) K^{\hat{Y}}_s(dx)dA^{\hat{Y}}_s.
\end{align*}
To gain the local supermartingale property we show that the process $Z = (Z_t)_{t \in [0,T]}$ defined by
\begin{align}
\begin{split}
\label{eq:defZ}
Z_t := &~\int_0^t \left[ \frac{\partial}{\partial s} G_f(s,Y_{s^-}) b^{\hat{Y}}_s + \frac{1}{2} \sum_{i,j \le d} \frac{\partial^2}{\partial x^i x^j} G_f(s,Y_{s^-}) c^{\hat{Y}ij}_s\right.\\
&~\left. + \int_{\RR^d} H_{G_f}(s,Y_{s^-},x) K^{\hat{Y}}_u(dx) \vphantom{\sum_{i,j \le d} \frac{\partial^2}{\partial x^i x^j}} \right]dA^{\hat{Y}}_u
\end{split}
\end{align}
is $Q_2$ almost surely non-increasing. Therefore, we use Assumption~$(ii)$ to replace the term with the time derivative $\frac{\partial}{\partial s} G_f(s,Y_{s^-}) b^{\hat{Y}}_s$. Note that since we consider the semimartingale characteristics under the particular e.m.m., we have $(b^{\hat{Y}} \cdot A^{\hat{Y}})_t = t = (b^{\hat{X}} \cdot A^{\hat{X}})_t$ for all $t \in [0,T]$. Consequently, we have by Assumption~$(v)$ that $b^{\hat{Y}}_t dA^{\hat{Y}}_t = dt = b^{\hat{X}}_t dA^{\hat{Y}}_t$. As consequence we obtain for $Z_t$
{\small
\begin{align*}
\int_0^t \left[\frac{1}{2} \sum_{i,j \le d} \frac{\partial^2}{\partial x^i x^j} G_f(s,Y_{s^-}) \left (c^{\hat{Y}ij}_s - c^{\hat{X}ij}_s\right ) + \int_{\RR^d} H_{G_f}(s,Y_{s^-},x) \left (K^{\hat{Y}}_s(dx) - K^{\hat{X}}_s(dx)\right ) \right]dA^{\hat{Y}}_s.
\end{align*}}
By Assumption~$(i)$, $G_f$ is directionally convex in the second entry, which is equivalent to the second partial derivatives to be nonnegative for all $i,j$, see \citet[Theorem 3.12.2]{MS02}. It follows with Assumption~$(vi)$ that the first integrand is non-positive $dA^{\hat{Y}} \times Q_2$ almost surely.

To see that the second integrand is non-positive $dA^{\hat{Y}} \times Q_2$ almost surely as well, we find that $H_{G_f}(s, Y_{s^-},x)$ is directionally convex in $x$. This follows from the directional convexity of $G_f$: 
\begin{align*}
\frac{\partial^2}{\partial x^i x^j} H_{G_f}(s, Y_{s^-},x) =&~ \frac{\partial^2}{\partial x^i x^j} \left(G_f(s, Y_{s^-} + x) - G_f(s, Y_{s^-}) - \sum_{k \le d} \frac{\partial}{\partial x^k} G_f(s, Y_{s^-}) x^k \right)\\
=&~ \partial^2_{ij} G_f(s, Y_{u^-} + x) \ge 0.
\end{align*}
With Assumption~$(vi)$ it follows that the second integrand is non-positive $dA^{\hat{Y}} \times Q_2$ almost surely. Therefore, $-Z \in \mathscr{A}_{loc} ^+$ and $(G_f(t, Y_t))_{t \in [0,T]}$ is a local $Q_2$-supermartingale.

From the fact that $G_f(0,Y_0) = E_{Q_1}[f(X_T)]$ is integrable and from Assumption~$(iv)$ it follows that $(G_f(t, Y_t))_{t \in [0,T]}$ is a proper supermartingale.

If the inequalities in $(v)$ are reversed and the positive part of $(G_f(t, Y_t))_{t \in [0,T]}$ is of class (DL), the process $Z$ is in $\mathscr{A}^+_{loc}$ and hence $G_f$ is a $Q_2$-submartingale.
\end{proof}

\begin{remark}[Comments on the assumptions in Theorem \ref{thm:orderingdcxemm}]
\label{rem:onthm:orderingdcxemm}
Conditions $(iii)$ and $(iv)$ of Theorem \ref{thm:orderingdcxemm} are clearly unavoidable since we need the proper supermartingale property and in order to compensate the jumps. We comment on the other assumptions, while in Section \ref{sec:discussion} we give a more detailed discussion of the regularity conditions in this paper.
\begin{enumerate}
\item The regularity assumption $(i)$ is crucial for the applicability of It\^{o}'s formula. It is a common assumption in financial mathematics for the computation of option prices in a Markovian model by the PIDE method. For example in \citet{CT04} there are conditions given for the regularity of the functional $G_f$ in exponential L\'{e}vy models. 

\item The assumption of directional convexity of $G_f$ in $(ii)$ provides the positivity of the second derivative of $G_f$. This is necessary for the tractability of the relevant terms in It\^{o}'s formula. It is in particular fulfilled if the propagation of directional convexity property holds, i.e. for all directional convex functions $f$ it holds that $G_f \in \cF_{dcx}$. This assumption is made e.g. in \citet{BR06,BR08}. It holds in particular for processes with independent increments and for diffusion processes.

\item Assumption~$(ii)$ allows us to obtain a connection between the differential characteristics. If $X$ is a Markovian semimartingale and $f(X_T) \in L^1(Q_1)$, then the process $(G_f(t,X_t))_{t \in [0,T]}$ is a $Q_1$-martingale if the measure $Q_1$ preserves the Markov property. This is a consequence of the Markov property,
\begin{align*}
G_f(t,X_t) = E_{Q_1}[f(X_T)|X_t] = E_{Q_1}[f(X_T)|\cF_t],
\end{align*}
which is a martingale by construction. Hence, we can apply Lemma \ref{lemma:kolmogorovbackwardemm} to $G_f$ and obtain that $U^X_t G_f(t,X_{t^-}) = 0$ $dA^{\hat{X}} \times Q_1$ almost surely for all $t \in [0,T]$. Typically the differential characteristics of Markovian semimartingales are of the form $a(t,X_t)$, see \citet{Ci80}. Hence, we get that for all $x \in \supp\left((Q_1)^{X_{t^-}} \right)$ it holds that $U^X_t G_f(t,x) = 0$ $dA^{\hat{X}} \times Q_1$ almost surely for all $t \in [0,T]$ as well. If a semimartingale $Y$ fulfills\/ $\supp \left((Q_2)^{Y_{t^-}} \right) \subset \supp \left((Q_1)^{X_{t^-}} \right)$, we deduce that
\begin{align*}
U^X_t G_f(t,Y_{t^-}) = 0 
\end{align*}
$dA^{\hat{X}} \times Q_2$ almost surely for all $t \in [0,T]$.\\
Note that we can replace $Q_1$ by $Q_2$ since we assumed that $Q_1 \sim Q_2$. So Markov processes with a ``big'' support are candidate processes for the semimartingale $X$ in Theorem \ref{thm:orderingdcxemm}.\\
In a Markovian framework, $G$ is the transition operator of $X$ applied to $f$, a well understood object. This suggests that Theorem \ref{thm:orderingdcxemm} is particularly suitable if the semimartingale $X$ is a Markov process w.r.t $Q_1$. In particular, previous results in literature are special cases of Theorem \ref{thm:orderingdcxemm}.

\item Assumption~$(v)$ seems at first glance to be a severe restriction. This condition is fulfilled for example when we compare It\^{o} processes or when $Y$ is a Girsanov transform of $X$. In the setting of this section we have $Q_2 \sim Q_1$ since we assumed both measures to be equivalent to $P$. So this theorem is in particular useful if we compare one semimartingale under different e.m.m. In that case also the comparison of the characteristic simplifies since the quadratic variation of the continuous martingale part is unchanged.

Further, the integrator in a good version of the semimartingale characteristics can be chosen more or less freely. Only existence, not uniqueness is stated in \citet[Proposition II.2.9]{JS03}. So if we consider two semimartingales $X$ and $Y$ under the same measure we can, analogously to equation \eqref{eq:Afromniceversion}, find a joint integrator for a good version, for example the process
\begin{align*}
A =&~ \sum_{i \le d} \var(B^{\hat{X}i}) + \sum_{i,j \le d} \var(C^{\hat{X}ij}) + (|x|^2 \wedge 1) \ast \nu^{\hat{X}} \\
&~+ \sum_{i \le d} \var(B^{\hat{Y}i}) + \sum_{i,j \le d} \var(C^{\hat{Y}ij}) + (|x|^2 \wedge 1) \ast \nu^{\hat{Y}}.
\end{align*}

\item The inequalities between the differential characteristics are the key for the comparison result. In the proof we can see that it suffices to check the inequality between the kernels for the function $H_{G_f}$ only. Instead of directional convexity we can use any function class $\cF$ such that $H_{G_f} \in \cF$ for an ordering of the kernels.

We emphasize that in general $K^{\hat{X}}$ is not the kernel of the semimartingale characteristics of $\hat{X}$ under $Q_2$. However, $c^{\hat{X}}$ is the process from the differential characteristics of $\hat{X}$ under $Q_2$ since we use equivalent measures. This follows from the Girsanov theorem. Effectively we do not compare the semimartingale characteristics of $\hat{X}$ and $\hat{Y}$ under $Q_2$. We compare under $Q_2$ the semimartingale characteristics we get under the particular e.m.m. If $X$ and $Y$ are already local martingales we compare the differential characteristics under the same measure $P$. This is a special case of the theorem above and we will discuss the comparison under $P$ in Section \ref{sec:compsm} in more detail.

\item We could also demand that the inequalities in $(vi)$ hold $Q_2$ almost surely for all $t \in [0,T]$. However, the choice of the product measure is more general, even if it seems more complicated at first glance.

\item The assumption that $X$ and $Y$ start in the same point can be easily achieved by shifting one of the semimartingales. Depending on the aim of the comparison of $X$ and $Y$ this might not be reasonable. Then we can replace this assumption by demanding 
\begin{align*}
G_f(0,y_0) \le G_f(0,x_0).
\end{align*}
\end{enumerate}
\end{remark}

Next we derive an ordering result when the functional $G_f$ is a convex function in $x$. Therefore, we use the positive semidefinite order for matrices, also called Loewner order. Remind that for $A,B \in \RR^{d \times d}$ $A$ is said to be smaller than $B$ in the positive semidefinite order, if the matrix $B-A$ is positive semidefinite, i.e. for all $x \in \RR$ it holds $x' (B-A) x \ge 0$. We write $A \le_{psd} B$ if $A$ is smaller than $B$ in this order.

The following convex comparison theorem extends Theorem 2.6 in \citet{BR06}.

\begin{theorem}[Convex comparison under e.m.m.]
\label{thm:orderingcxemm}
Let $X, Y$ be semimartingales and let $X_0 = Y_0 = x_0 \in \RR^d$ almost surely. Let $f \in L^1 \left((Q_1)^{X_T} \right) \cap L^1 \left((Q_2)^{Y_T} \right) $ and assume that
\begin{itemize}
\item[(i)] $G_f \in C^{1,2}([0,T] \times \RR^d)$ and $G_f(t, \cdot) \in \cF_{cx}$ for all $t \in [0,T]$;

\item[(ii)] - $(v)$ of Theorem \ref{thm:orderingdcxemm} hold;

\item[(vi)] The differential characteristics are $dA^{\hat{Y}} \times Q_2$ almost surely ordered:
\begin{align*}
c_t^{\hat{Y}} \le_{psd} &~ c_t^{\hat{X}},\\
\int_{\RR^d} g(t,Y_{t^-},x) K^{\hat{Y}}_t(dx) \le&~ \int_{\RR^d} g(t,Y_{t^-},x) K^{\hat{X}}_t(dx), 
\end{align*}
where the second inequality holds for all $g(t,y,\cdot) \in \cF_{cx}$ such that the integrals exist.
\end{itemize} 

Then it holds
\begin{align*}
E_{Q_2}[f(Y_T)] \le E_{Q_1}[f(X_T)].
\end{align*}
If the inequalities in $(vi)$ are reversed and $(G_f(t,Y_t)^+)_{t \in [0,T]}$ is of class (DL), we get
\begin{align*}
E_{Q_2}[f(Y_T)] \ge E_{Q_1}[f(X_T)].
\end{align*}
\end{theorem}

\begin{proof}
We show that $(G_f(t,Y_t))_{t \in[0,T]}$ is a $Q_2$-supermartingale. Similarly as in the proof of Theorem \ref{thm:orderingdcxemm} we need to show, that the process
{\small
\begin{align*}
\int_0^t \left[\frac{1}{2} \sum_{i,j \le d} \frac{\partial^2}{\partial x^i x^j} G_f(s,Y_{s^-}) \left (c^{\hat{Y}ij}_s - c^{\hat{X}ij}_s\right ) + \int_{\RR^d} H_{G_f}(s,Y_{s^-},x) \left (K^{\hat{Y}}_s(dx) - K^{\hat{X}}_s(dx)\right ) \right]dA^{\hat{Y}}_s
\end{align*}}
is non-increasing $dA^{\hat{Y}} \times Q_2$ almost surely. By Assumption~$(vi)$ the matrix $-(c^{\hat{Y}}_t - c^{\hat{X}}_t) = c_t^{\hat{X}} - c_t^{\hat{Y}}$ is positive semidefinite for fixed $(\omega,t)$. Thus, the eigendecomposition has the form $(\sum_{k \le d} \lambda_k e_k^i e_k^j)_{i,j \le d}$ with eigenvalues $\lambda_k \ge 0$ and eigenvectors $e_k$. We get that the first integrand has the form 
\begin{align*}
- \frac{1}{2} \sum_{k \le d} \lambda_k \sum_{i,j \le d} \frac{\partial^2}{\partial x^i x^j} G_f(s,Y_{s^-}) e^i_k e^j_k = - \frac{1}{2} \sum_{k \le d} \lambda_k e'_k \frac{\partial^2}{\partial x^i x^j} G_f(s,Y_{s^-}) e_k
\end{align*}
which is non-positive $dA^{\hat{Y}} \times Q_2$ almost surely due to the positive semidefiniteness of the Hessian matrix of $G_f$. 

Analogously to the proof of Theorem \ref{thm:orderingdcxemm} we have that $H_{G_f}(s,Y_{s^-},x)$ is convex in $x$ since the second derivative in direction of $x$ is 
\begin{align*}
\frac{\partial^2}{\partial x^i x^j} H_{G_f}(s,Y_{s^-},x) =&~ \frac{\partial^2}{\partial x^i x^j} \left(G_f(s, Y_{s^-} + x) - G_f(s, Y_{s^-}) - \sum_{k \le d} \frac{\partial}{\partial x^k} G_f(s, Y_{s^-}) x^k \right)\\
=&~ \frac{\partial^2}{\partial x^i x^j} G_f(s, Y_{s^-} + x).
\end{align*}
Therefore, the Hessian matrix of $H_{G_f}$ is positive semidefinite and it follows that $H_{G_f}$ is convex in $x$. Consequently, the second integrand is non-positive $dA^{\hat{Y}} \times Q_2$ almost surely by Assumption~$(vi)$. By Assumption~$(iv)$ it follows that $(G_f(t,Y_t))_{t \in[0,T]}$ is a proper supermartingale.

If the inequalities in $(vi)$ are reversed and $(G_f(t,Y_t)^+)_{t \in [0,T]}$ is of class (DL), then $(G_f(t,Y_t))_{t \in[0,T]}$ is a submartingale.
\end{proof}

\begin{remark}
\label{rem:cons}
As seen in the proofs, the key inequality is
\begin{align*}
\frac{1}{2} \sum_{i,j \le d} \frac{\partial^2}{\partial x^i x^j} G_f(s,Y_{s^-}) \left (c^{\hat{Y}ij}_s - c^{\hat{X}ij}_s\right ) + \int_{\RR^d} H_{G_f}(s,Y_{s^-},x) \left (K^{\hat{Y}}_u(dx) - K^{\hat{X}}_u(dx)\right ) \le 0
\end{align*}
$dA^{\hat{Y}} \times Q_2$ almost surely. Thus, we can replace the ordering assumption on the semimartingale characteristics by this inequality. Then also the (directional) convexity of $G_f$ is not necessary anymore. 

This is a starting point for ordering results of other function classes (\citet{BR07a}). Based on this inequality and under the assumption of propagation of order, there is given a table with conditions on the semimartingale characteristics for the comparison of further function classes. The classes investigated therein are increasing, supermodular, convex and directionally convex as well as increasing supermodular, increasing convex and increasing directionally convex functions.
\end{remark}

The considerations in Remark \ref{rem:cons} lead to the following corollary giving a comparison result under more general conditions on $G_f$.

\begin{corollary}[general comaprison under e.m.m.]
\label{cor:orderingoassemm}
Let $X, Y$ be semimartingales and let $X_0 = Y_0 = x_0 \in \RR^d$ almost surely. Let $f \in L^1 \left((Q_1)^{X_T} \right) \cap L^1 \left((Q_2)^{Y_T} \right)$ and assume that $G_f \in C^{1,2}([0,T] \times \RR^d)$ and that $(ii)$--$(v)$ of Theorem \ref{thm:orderingdcxemm} hold. Further, let $dA^{\hat{Y}} \times Q_2$ almost surely
\begin{align}
\label{eq:orderingnoass1emm}
\frac{1}{2} \sum_{i,j \le d} \partial^2_{ij} G_f(s,Y_{s^-}) \left (c^{\hat{Y}ij}_s - c^{\hat{X}ij}_s\right ) + \int_{\RR^d} H_{G_f}(s,Y_{s^-},x) \left (K^{\hat{Y}}_s(dx) - K^{\hat{X}}_s(dx)\right ) \le 0.
\end{align}
Then it holds that 
\begin{align}
\label{eq:orderingnoassemm}
E_{Q_1}[f(X_T)] \le E_{Q_2}[f(Y_T)].
\end{align}
If the term in \eqref{eq:orderingnoass1emm} is non-negative and $(G_f(t,Y_t)^+)_{t \in [0,T]}$ is of class (DL), we get the reverse inequality in \eqref{eq:orderingnoassemm}.
\end{corollary}

\begin{proof}
Equation \eqref{eq:orderingnoass1emm} is chosen in such a way that the process $Z$ defined in \eqref{eq:defZ} is non-increasing or non-decreasing and the assertion follows as in Theorem \ref{thm:orderingdcxemm}.
\end{proof}

\begin{remark}
Equation \eqref{eq:orderingnoass1emm} arises in a similar form in the field of model uncertainty in financial mathematics under the notion volatility misspecification, see \citet{KJS98}. There it is assumed that a market participant uses a model to price and hedge a European option which does not coincide with the real evolution of the underlying. Then the left side of inequality \eqref{eq:orderingnoass1emm} indicates the so-called tracking error which is the difference of the real price of the option and the price derived by the model of the market participant.
\end{remark}

As stated in Remark \ref{rem:onthm:orderingdcxemm} these kind of theorems are especially useful if we compare a single semimartingale under different equivalent martingale measures We now turn to this special case.

By Girsanov's theorem only the compensator of the jump measure changes while the predictable quadratic variation of the continuous martingale part and the increasing process of a good version of the semimartingale characteristics remain the same.

\begin{corollary}[Comparison of e.m.m.]
\label{cor:girsanovemm}
Let $X$ be a semimartingale, let $Q_1$ and $Q_2$ be equivalent martingale measures for $X$ and denote the particular semimartingale characteristics of $X$ by superscript. Assume that $f \in L^1 \left((Q_1)^{X_T} \right) \cap L^1 \left((Q_2)^{X_T} \right)$ and that
\begin{enumerate}
\item[(i)] $G_f \in C^{1,2}([0,T] \times \RR^d)$ and $G_f(t,\cdot) \in \cF_{dcx}$ (or $G_f(t,\cdot) \in \cF_{cx}$) for all $t \in [0,T]$;

\item[(ii)] $U_t^X G_f(t,X_{t^-}) = 0$ $dA^{\hat{X}} \times Q_1$ almost surely where $U_t^X$ is defined in \eqref{eq:kolmogorovbackwardemm}, here with semimartingale characteristics of $X$ under $Q_2$;

\item[(iii)] $\big| H_{G_f} \big| \ast \mu^X \in \mathscr{A}_{loc}^+$;

\item[(iv)] $(G_f(t, X_t)^-)_{t \in [0,T]}$ is of class (DL);

\item[(v)] The kernels $K^1$ and $K^2$ are $dA^{\hat{X}} \times Q_1$ almost surely ordered for all $t \in [0,T]$:
\begin{align*}
\int_{\RR^d} g(t,X_{t^-},x) K^1_t(dx) \le&~ \int_{\RR^d} g(t,X_{t^-},x) K^2_t(dx), 
\end{align*}
where the inequality holds for all $g(t,y,\cdot) \in \cF_{dcx}$ (or $g(t,y,\cdot) \in \cF_{cx}$) such that the integrals exist.
\end{enumerate} 

Then we obtain
\begin{align*}
E_{Q_1}[f(X_T)] \le E_{Q_2}[f(X_T)].
\end{align*}
If the inequalities in $(vi)$ are reversed and $(G_f(t, X_t)^+)_{t \in [0,T]}$ is of class (DL), we get 
\begin{align*}
E_{Q_1}[f(X_T)] \ge E_{Q_2}[f(X_T)].
\end{align*}
\end{corollary}

\begin{proof}
This follows directly from Theorem \ref{thm:orderingdcxemm} and Theorem \ref{thm:orderingcxemm}.
\end{proof}

\begin{remark}
\label{rem:oncoremm}
As remarked in Corollary \ref{cor:orderingoassemm} we do not need the assumption of convexity and directional convexity of $G_f$ since the terms with second partial derivatives vanish in It\^{o}'s formula. This is due to the fact that the semimartingale characteristic from the continuous martingale part $c$ remains the same when we change the measure. Assumption~$(v)$ then needs to be adapted to $H_{G_f}$ as in Corollary \ref{cor:orderingoassemm}. Having in mind that it suffices to have inequality of the kernels for $H_{G_f}$, the result also follows by means of It\^{o}'s formula. After the partial derivative in time is replaced, the only remaining term of finite variation then is non-increasing by assumption $(v)$.
\end{remark}

\section{Comparison under the same semimartingale measure}
\label{sec:compsm}

We turn in this section to the case, when we regard both semimartingales under $P$. As in Section \ref{sec:funceq} we restrict ourselves to special semimartingales.

We begin with a version of the comparison of processes $X$ and $Y$ in Theorem \ref{thm:orderingdcxemm} under the same semimartingale measure $P$. Let $X$ and $Y$ be special semimartingales, then the processes $\hat{X}$ and $\hat{Y}$ are special semimartingales and we can choose for both semimartingales the same process $A$ for a good version of the semimartingale characteristics under $P$, see Remark \ref{rem:onthm:orderingdcxemm}. Further, we choose the identity as truncation function. The canonical decomposition of $\hat{X}$ has the form:
\begin{align*}
\hat{X}_t = \hat{X}_0 + \left(0,X_t^c + x \ast \left (\mu^X - \nu\right )_t \right) + (b^{\hat{X}} \cdot A)_t,
\end{align*}
Analogously we have such a decomposition for $\hat{Y}$ and we use superscripts to point out to which process the characteristics belong. The functional $G_f$ is defined in equation \eqref{eq:backwardfunctional} as conditional expectation under the equivalent martingale measure $Q_1$ and needs to be adapted. We define the valuation operator 
\begin{align*}
G_f (t,x) := E[f(X_T)|X_t=x],
\end{align*}
where the conditional expectation is now with respect to $P$. Since in this setting the drift part is not only the identity we need to control additionally the first derivatives in It\^o's formula. We accomplish this by assuming that $G_f(t,\cdot)$ is an increasing function for all $t \in [0,T]$ in the poitwise ordering on $\RR^d$.

\begin{theorem}[Increasing directionally convex comparison under $P$]
\label{thm:orderingidcxp}
Let $X, Y$ be special semimartingales and let $X_0 = Y_0 = x_0$ almost surely and let $f \in L^1(P^{X_T}) \cap L^1 (P^{Y_T})$. Assume that
\begin{itemize}
\item[(i)] $G_f \in C^{1,2}([0,T] \times \RR^d)$ and $G_f(t, \cdot) \in \cF_{idcx}$ for all $t \in [0,T]$;

\item[(ii)] $\bar{U}^X_t G_f(t,Y_{t^-}) = 0$ holds $dA \times P$ almost surely for all $t \in [0,T]$ where $\bar{U}^X_t$ is defined in \eqref{eq:kolmogorovbackwardp} with the characteristic of $\hat{X}$ in it;

\item[(iii)] $\big| H_{G_f} \big| \ast \mu^Y \in \mathscr{A}_{loc}^+$;

\item[(iv)] $(G_f(t, Y_t)^-)_{t \in [0,T]}$ is of class (DL);

\item[(v)] The differential characteristics are $dA \times P$ almost surely ordered for all $i,j \le d$:
\begin{align*}
b_t^{\hat{Y} i} \le &~b_t^{\hat{X} i},\\
c_t^{\hat{Y} ij} \le &~ c_t^{\hat{X} ij},\\
\int_{\RR^d} g(t,Y_{t^-},x) K^{\hat{Y}}_t(dx) \le&~ \int_{\RR^d} g(t,Y_{t^-},x) K^{\hat{X}}_t(dx), 
\end{align*}
where the last inequality holds for all $g(t,y,\cdot) \in \cF_{idcx}$ such that the integrals exist.
\end{itemize}
Then it holds
\begin{align*}
E[f(Y_T)] \le E[f(X_T)].
\end{align*}
If the inequalities in $(v)$ are reversed and $(G_f(t, Y_t)^+)_{t \in [0,T]}$ is of class (DL), we obtain that
\begin{align*}
E[f(Y_T)] \ge E[f(X_T)].
\end{align*}
\end{theorem}

\begin{proof}
Analogously to the comparison under equivalent martingale measures we show that $(G_f(t,Y_t))_{t \in [0,T]}$ is a $P$-supermartingale. It\^{o}'s formula yields
{\small
\begin{align*}
G_f(t, Y_t) =&~ G_f(0,x_0) + \int_0^t \frac{\partial}{\partial s} G_f(s,Y_{s^-}) b^{\hat{Y}}_s dA_s + \sum_{i \le d} \int_0^t \frac{\partial}{\partial x^i} G_f(s, Y_{s^-}) dY_s^i\\
&~ + \frac{1}{2} \sum_{i,j \le d} \int_0^t \frac{\partial^2}{\partial x^i x^j} G_f(s,Y_{s^-}) c^{\hat{Y}ij}_s dA_s\\ 
&~ + \int_{[0,t] \times \RR^d} \left[ G_f(s, Y_{s^-} + x) - G_f(s, Y_{s^-}) - \sum_{i \le d} \frac{\partial}{\partial x^i} G_f(s, Y_{s^-}) x^i \right] \mu^Y (ds,dx).
\end{align*}}
We compensate the jumps and use the canonical decomposition of $\hat{Y}$. This leads to
\begin{align*}
G_f(t, Y_t) =&~ G_f(0,x_0) + \int_0^t \frac{\partial}{\partial s} G_f(s,Y_{s^-}) b^{\hat{Y}}_s dA_s + \sum_{i \le d} \int_0^t \frac{\partial}{\partial x^i} G_f(s, Y_{s^-}) b^{\hat{Y} i}_s dA_s\\
&~ + \frac{1}{2} \sum_{i,j \le d} \int_0^t \frac{\partial^2}{\partial x^i x^j} G_f(s,Y_{s^-}) c^{\hat{Y}ij}_s dA_s + M_t \\
&~+ \int_{[0,t] \times \RR^d} H_{G_f}(s,Y_{s^-},x) K^{\hat{Y}}_s(dx)dA_s,
\end{align*}
where $M$ is the local martingale from the integrals with respect to the continuous martingale part of $Y$ and the compensated jumps. With similar arguments as in the proof of Theorem \ref{thm:orderingdcxemm} it suffices to show, that the following process $Z_t$ is non-increasing $dA \times P$ almost surely:
\begin{align*}
Z_t :=&~ \int_0^t \sum_{i \le d} \frac{\partial}{\partial x^i} G_f(s, Y_{s^-}) \left(b^{\hat{Y} i}_s - b^{\hat{X}i}_s \right) + \frac{1}{2} \sum_{i,j \le d} \frac{\partial^2}{\partial x^i x^j} G_f(s,Y_{s^-}) \left (c^{\hat{Y}ij}_s - c^{\hat{X}ij}_s\right ) \\ 
&~ + \int_{\RR^d} H_{G_f}(s,Y_{s^-},x) \left (K^{\hat{Y}}_s(dx) - K^{\hat{X}}_s(dx)\right ) dA_s.
\end{align*}
The first term in the integral is non-positive $dA \times P$ almost surely because of Assumption~$(v)$ and the fact that $G_f$ is increasing in $x$. The remaining terms are non-positive $dA \times P$ almost surely which can be seen as in the proof of Theorem \ref{thm:orderingdcxemm}. Assumption~$(iv)$ yields the proper supermartingale property.\\
If the inequalities in $(v)$ are reversed and $(G_f(t, Y_t)^+)_{t \in [0,T]}$ is of class (DL), the process $(G_f(t,Y_t))_{t \in [0,T]}$ is a submartingale.
\end{proof}

\begin{remark}
\begin{enumerate}
\item[1.]
Since we assumed the functional $G_f$ to be increasing in the second variable it is not necessary anymore that $Y_0 = X_0$. If $X_0 = x_0 \ge y_0 = Y_0$, the supermartingale property and the fact that $G_f$ is increasing in the second variable still yield the inequality of the expectations, 
\begin{align*}
E[f(Y_T)] = E[G_f(T,Y_T)] \le G_f(0,y_0) \le G_f(0,x_0) = E[f(X_T)].
\end{align*}
When $G_f$ is a submartingale, we need to impose $X_0 = x_0 \le y_0 = Y_0$ for an analog statement.

\item[2.] 
In contrast to the last section we compare in this framework the original semimartingale characteristics, cf. Remark \ref{rem:onthm:orderingdcxemm}.
\end{enumerate}
\end{remark}

The following comparison result for processes $X$ and $Y$ in the case that $G_f$ is increasing and convex in $x$ is the analogon of the comparison result in Theorem \ref{thm:orderingcxemm} under a semimartingale measure $P$.

\begin{theorem}[Increasing convex comparison under P]
\label{thm:orderingicxp}
Let $X, Y$ be special semimartingales and let $x_0 = X_0 \ge Y_0 = y_0$ almost surely. Let $f \in L^1 \left (P^{X_T}\right ) \cap L^1\left (P^{Y_T}\right )$ and assume that
\begin{itemize}
\item[(i)] $G_f \in C^{1,2}([0,T] \times \RR^d)$ and $G_f(t, \cdot) \in \cF_{icx}$ for all $t \in [0,T]$;

\item[(ii)] - $(iv)$ of Theorem \ref{thm:orderingidcxp} hold;

\item[(v)] The differential characteristics are $dA \times P$ almost surely ordered for all $t \in [0,T]$ and all $i \le d$:
\begin{align*}
b_t^{\hat{Y} i} \le &~b_t^{\hat{X} i},\\
c_t^{\hat{Y}} \le&_{psd} ~ c_t^{\hat{X}},\\
\int_{\RR^d} g(t,Y_{t^-},x) K^{\hat{Y}}_t(dx) \le&~ \int_{\RR^d} g(t,Y_{t^-},x) K^{\hat{X}}_t(dx), 
\end{align*}
where the last inequality holds for all $g(t,y,\cdot) \in \cF_{cx}$ such that the integrals exist.
\end{itemize} 

Then we have that
\begin{align*}
E[f(Y_T)] \le E[f(X_T)].
\end{align*}
If in $(v)$ the inequalities are reversed, $x_0 = X_0 \le Y_0 = y_0$ and $(G_f(t, Y_t)^+)_{t \in [0,T]}$ is of class (DL), we get
\begin{align*}
E[f(Y_T)] \ge E[f(X_T)].
\end{align*} 
\end{theorem}

\begin{proof}
We show, that the process
\begin{align*}
\int_0^t &\left[\sum_{i \le d} \frac{\partial}{\partial x^i} G_f(s, Y_{s^-}) \left(b^{\hat{Y} i}_s - b^{\hat{X}i}_s \right) + \frac{1}{2} \sum_{i,j \le d} \frac{\partial^2}{\partial x^i x^j} G_f(s,Y_{s^-}) \left (c^{\hat{Y}ij}_s - c^{\hat{X}ij}_s\right ) \right.\\
& + \left. \int_{\RR^d} H_{G_f}(s,Y_{s^-},x) \left (K^{\hat{Y}}_s(dx) - K^{\hat{X}}_s(dx)\right ) \vphantom{\sum_{i,j \le d} \frac{\partial^2}{\partial x^i x^j}} \right]dA^{\hat{Y}}_s
\end{align*}
is non-increasing $P$ almost surely. Then the assertion follows as in Theorem \ref{thm:orderingidcxp}. The first term in the integral is non-positive due to Assumption~$(v)$ and the fact that $G_f$ is increasing in the second variable. The remaining part is non-positive similarily as in the proof of Theorem \ref{thm:orderingcxemm}. By Assumption~$(iv)$ it follows that $(G_f(t,Y_t))_{t \in[0,T]}$ is a proper supermartingale.\\
If the inequalities in $(v)$ are reversed, $x_0 = X_0 \le Y_0 = y_0$ and $(G_f(t, Y_t)^+)_{t \in [0,T]}$ is of class (DL), it is a submartingale.
\end{proof}

As under e.m.m., the key inequality of the proof can be used to replace the convexity assumption by a general form of conditions.

\begin{corollary}[General comparison condition]
Let $X, Y$ be special semimartingales, let $x_0 = X_0 \ge Y_0 = y_0$ and let $f \in L^1 \left (P^{X_T}\right ) \cap L^1\left (P^{Y_T}\right )$. Assume that $G_f \in C^{1,2}([0,T] \times \RR^d)$ and that assumptions $(ii)$--$(iv)$ of Theorem \ref{thm:orderingidcxp} hold. Further, assume that $dA \times P$ almost surely
\begin{align}
\begin{split}
\label{eq:orderingnoass1p}
&~\sum_{i \le d} \frac{\partial}{\partial x^i} \cG_f(s, Y_{s^-}) \left(b^{\hat{Y} i}_s - b^{\hat{X}i}_s \right) + \frac{1}{2} \sum_{i,j \le d} \frac{\partial^2}{\partial x^i x^j} G_f(s,Y_{s^-}) \left (c^{\hat{Y}ij}_s - c^{\hat{X}ij}_s\right ) \\
&~ + \int_{\RR^d} H_{G_f}(s,Y_{s^-},x) \left (K^{\hat{Y}}_s(dx) - K^{\hat{X}}_s(dx)\right ) \le 0.
\end{split}
\end{align}
Then it holds that 
\begin{align*}
E[f(X_T)] \le E[f(Y_T)].
\end{align*}
If the inequality \eqref{eq:orderingnoass1p} is reversed, $x_0 = X_0 \le Y_0 = y_0$ and $(G_f(t, Y_t)^+)_{t \in [0,T]}$ is of class (DL), we get $E[f(X_T)] \ge E[f(Y_T)]$.
\end{corollary}

\begin{proof}
Similarly as in the proof of Theorem \ref{thm:orderingicxp} we obtain that under these assumptions $(G_f(t,Y_t))_{t \in [0,T]}$ is a supermartingale (submartingale if the inequality is inverse). 
\end{proof}

\section{Discussion of assumptions and examples}
\label{sec:discussion}

In this section we describe several approaches to establish the regularity conditions of the comparison results and give examples. The focus is on the question of differentiability and convexity. It is clear that the regularity and (directional) convexity condition on $G_f$ only depend on the semimartingale which is used in the definition of $G_f$.

For probability measures $P$ and $Q$ on a space $(\Omega, \cF)$ and a class of integrable functions $\mathfrak{F}$, $P$ is said to be smaller than $Q$ in the \emph{integral stochastic order generated by $\mathfrak{F}$}, $P \preceq_{\mathfrak{F}} Q$ if
\begin{align*}
\int f dP \le \int f dQ,~ \text{for all}~ f \in \mathfrak{F}~.
\end{align*}
So in terms of integral stochastic orders we state in Section \ref{sec:compemm} conditions for the ordering $\left(Q_2\right)^{Y_T} \preceq_{\mathfrak{F}} \left(Q_1\right)^{X_T}$ and in Section \ref{sec:compsm} for the ordering $P^{Y_T} \preceq_{\mathfrak{F}} P^{X_T}$ for some function class $\mathfrak{F}$. The question arises what function classes are well fitting with theses conditions. So far we did not impose conditions on the function $f$ under consideration but only on the functional $G_f$.

The section is organized as follows. We first discuss approaches from the theory of Markov processes to conclude differentiability of $G_f$. In particular for smooth functions a direct argument for differentiability can be given. Therefore, we recapitulate some insights from the PIDE method in option pricing. Then we proceed with another ansatz in the framework of integral stochastic orders. Afterwards we deal with the issue of convexity and directional convexity of $G_f$ and in particular remind some approaches from the literature and give corresponding references. We conclude this section with some explicit examples of processes which can be compared with the theorems of this paper.

\subsection{Differentiability}
\label{subsubsec:regularity}

There are various approaches in the literature to establish the regularity assumptions on $G_f$. In general, the representation
\begin{align}
\label{eq:G_g_ausgeschrieben}
G_f(t,x) = E[f(X_T)|X_t=x] = \int_{\RR^d} f(y) P^{X_T|X_t =x}(dy).
\end{align}
suggests that differentiability of $G_f$ is mainly an issue of the conditional distribution. For Markov processes this question is a well studied object. Also for the computation of option prices this question has been investigated in many papers in particular in connection with the PIDE method. For L\'evy processes differentaibility can be shown by a convolution argument (see \citet{CT04}, \citet{Gl10}). Assume that a L\'evy process $X_t$ possesses a density $p_t$ with respect to the Lebesgue measure. Due to the temporal and spatial homogeneity of L\'{e}vy processes we can simplify the conditional expectation
\begin{align}
\begin{split}
\label{eq:G_gconvolutionLevy}
G_f(t,x) =&~ E[f(X_T)|X_t = x] \\
=&~ E[f(X_{T-t} +x)]\\
=&~ \int_{\RR^d} f(y +x) p_{T-t}(y) dy\\
=&~ \int_{\RR^d} \tilde{f}(- x - y) p_{T-t}(y) dy\\
=&~ \tilde{f} \ast p_{T-t}(-x),
\end{split}
\end{align} 
where $\tilde{f}(x) = f(-x)$. If the density $p$ is twice continuously differentiable in $x$ so is $G_f$. In \citet[Proposition 3.12]{CT04} conditions are given so that the density of a L\'{e}vy process is smooth. More generally we can check the number of times of continuous differentiability of a density with the help of its associated characteristic function, see \citet[Proposition 28.1]{Sa99}. If for the characteristic function $\hat{\mu}$ of a measure $\mu$ it holds that $\int_{\RR^d} |z|^n |\hat{\mu}(z)| dz < \infty$ for a $n \in \NN$, then $\mu$ has a $C^n$ Lebesgue density.

For smooth functions satisfying some Lipschitz conditions, smoothness of $G_f$ can be shown directly (without smoothness of the density). For a smooth Lipschitz continuous function $f$, we obtain by dominated convergence
\begin{align*}
\frac{\partial}{\partial x^i} G_f(t,x) =&~ \lim_{h \to 0} \frac{E[f(X_{T-t} + x + h e_i)- f(X_{T-t} +x)]}{h}\\
=&~ E\left[ \lim_{h \to 0} \frac{f(X_{T-t} + x + h e_i)- f(X_{T-t} + x)}{h} \right]\\
=&~ E\left [ \frac{\partial}{\partial x^i} f(X_{T-t} - x)\right ].
\end{align*}
Similarly we get $\frac{\partial^2}{\partial x^i x^j} G_f(t,x) = E[ \frac{\partial^2}{\partial x^i x^j} f(X_T - x)]$ if the derivative of $f$ is Lipschitz continuous. This is continuous if $\frac{\partial^2}{\partial x^i x^j} f$ is Lipschitz continuous.

In this conection it is of interest that in the theory of stochastic orders it has been established that several stochastic orders can be generated by classes of smooth functions, see \citet{MS02}. Therefore, in the framework of L\'evy processes, we get directly from equation \eqref{eq:G_gconvolutionLevy} that $G_f$ is smooth in $x$ with no further assumptions on the density. In particular, the directionally convex order and increasing directionally convex order are integral stochastic orders which are generated by smooth functions.

For the differentiability in $t$ we consider a time-homogeneous Markov process $X$. Assume that its transition operators $(T_t)_{t\ge 0}$ form a strongly continuous semigroup. Then we have $G_f(t,\cdot) = T_{T-t}f(\cdot)$ and by semigroup theory $G_f$ is differentiable in $t$. The continuity of the derivative then follows because the semigroup of a Markov process $(T_t)_{0 \le t \le T}$ with generator $A$ can be represented as solution to an evolution problem( c.f. \citet{EK05}), i.e. for $f \in \cD(A)$ such that $Af \in \cD(A)$ it holds:
\begin{align}
\label{eq:semigroupcontinuityGg}
\frac{d}{dt} T_t f = T_t A f.
\end{align}
The term $Af$ does not depend on time and $T_t$ is assumed to be continuous. In the time-inhomogeneous case we can achieve continuous differentiability in time analogously with the theory of evolution systems. Then we obtain an equation similar to \eqref{eq:semigroupcontinuityGg}.

In the case of L\'{e}vy processe, this holds for functions in $C_0(\RR^d)$, i.e. continuous functions vanishing at infinity. For these functions the transition operators of a L\'{e}vy process form a strongly continuous semigroup, see \citet[Theorem 31.5]{Sa99} .

\subsection{Convexity and directional convexity}

Concerning convexity and directional convexity the \emph{propagation of order property} for a function class $\mathfrak{F}$ has turned out to be useful, i.e. $f \in  \mathfrak{F}$ implies $G_f \in \mathfrak{F}$. Therefore properties for $G_f$, as for example convexity, can be inferred from those of $f$. For papers in this field see e.g. \citet{Bm96}, \citet{Ma99}, \citet{KJS98}, \citet{BJ00}, \citet{GM02} and \citet{BR06,BR07a}. The results in this direction use various methods as establishing a Cauchy problem  for $G_f$ in the case of one-dimensional diffusions and reduction to the non-crossing property by the Feynman--Kac formula (see \citet{Bm96} and \citet{KJS98}) or using independent increments as \citet{GM02}. In \citet{BR06,BR07a} an approximation argument is used. A coupling aproach is used in \citet{Ho98}. A detailed exposition of these approaches is given in \citet{Ko19}

\subsection{Concrete examples}

In the sequel we give some explicit examples for semimartingales which fulfill the conditions of the comparison results. Throughout this section we assume that $G_f$ is continuously differentiable in time and focus on the differentiability in space.

We begin with the assumption that $X$ is a one-dimensional L\'{e}vy process. Then by the considerations above we conclude that convexity and directional convexity is propagated. Also the differentiability in time is straightforward if we restrict ourselves to functions in $C_0(\RR^d)$. In addition we can use a result from \citet{Sa99} which characterises the support of a one-dimensional L\'{e}vy process.

\begin{example}[Theorems \ref{thm:orderingidcxp} and \ref{thm:orderingicxp} for $\RR$-valued L\'{e}vy processes and It\^{o} semimartingales]
\label{ex:exampleorderingp}
We consider a comparison under $P$. Let $X$ be an $\RR$-valued L\'{e}vy process with differential semimartingale characteristic $(b,c^2,K)$, where $b \in \RR$, $c \in \RR_+$ and $K$ is a L\'{e}vy measure. We assume that $X$ is a type C L\'{e}vy process, i.e. either $c^2 > 0$ or $\int_{|x| \le 1} |x| K(dx) = \infty$. Additionally the L\'{e}vy measure is assumed to fulfill
\begin{align}
\label{eq:levyprocessspecial}
\int_{|x|>1} |x| K(dx) < \infty
\end{align}
and 
\begin{align}
\label{eq:existencesmoothdensity}
\liminf_{\varepsilon \downarrow 0} \varepsilon^{-c} \int_{-\varepsilon}^{\varepsilon} |x|^2 K(dx) > 0,
\end{align}
where $c \in (0,2)$. The first condition assures that $X$ is a special semimartingale, see \citet[Lemma 4.5]{RS11}. The second condition yields the existence of a smooth density, see \citet[Proposition 3.12]{CT04}.

Further, let $Y = (Y_t)_{0 \le t \le T}$ be a special It\^{o} semimartingale in the sense of \citet{Ci80}, i.e. a special semimartingale of the form 
\begin{align*}
Y_t = y_0 + \int_0^t \beta_s ds + \int_0^t \delta_s dB_s + \int_0^t \int_\RR y \tilde{\mu}^Y(ds,dy).
\end{align*}
Here $y_0 \in \RR$, $\beta$ and $\delta$ are adapted processes such that the integral exist, $(B_t)_{0 \le t \le T}$ is a standard Brownian motion and $\tilde{\mu}^Y$ is the compensated jump measure of $Y$. The compensator is of the form $\nu(dt,dx) = dt~ n_t(dx)$.\\
It follows that the identity is an integrator for a good version of the semimartingale characteristics for $X$ and $Y$. Also we can use the identity in the good version of $\hat{X}$ and $\hat{Y}$ and do not need to adapt it as in Remark \ref{rem:onthm:orderingdcxemm}. Besides the differential characteristics of $X$ and $Y$ are the differential characteristics which occur in the space dimensions of $\hat{X}$ and $\hat{Y}$ respectively. We consider an integrable increasing convex or increasing directionally convex function $f$.

From the independence of increments of $X$ we obtain propagation of increasing convexity and increasing directional convexity. By \eqref{eq:existencesmoothdensity} $X$ possesses a smooth Lebesgue density and it follows that $G_f \in C^{1,2}$. Thus, condition $(i)$ of Theorems \ref{thm:orderingidcxp} and \ref{thm:orderingicxp} is fulfilled.

From the Markov property of $X$ we get that the functional $G_f$ can be written as
\begin{align*}
G_f(t,X_t) = E[f(X_T)|X_t] = E[f(X_T)|\cF_t],
\end{align*}
which is a martingale by construction. Therefore, Lemma \ref{lemma:kolmogorovbackwardp} can be applied. Further, since $X$ is a type C L\'{e}vy process, we have by \citet[Theorem 24.10]{Sa99} that $\supp(P^{X_t}) = \RR$ for all $t \ge 0$. Hence, we get for all $x \in \RR$ and all $t \in [0,T]$ that
\begin{align*}
\bar{U}^X_t G_f(t,x) = 0,
\end{align*}
$dt \times P$ almost surely. It follows that
\begin{align*}
\bar{U}^X_t G_f(t,Y_{t^-}) = 0,
\end{align*}
$dt \times P$ almost surely for all $t \in [0,T]$. Altogether condition $(ii)$ is fulfilled. 
This holds independently from the choice of $Y$, see Remark \ref{rem:onthm:orderingdcxemm}. This consideration shows that type C L\'{e}vy processes are candidates for the semimartingale $X$ in the theorems stated in this chapter. Further, we assume the conditions $(iii)$ and $(iv)$. 

If now the differential characteristics are $dt \times P$ almost surely ordered,
\begin{align*}
\beta_t \le&~ b\\
\delta^2_t \le&~ c\\
\int_{\RR} g(t,Y_{t^-},x) \nu^Y_t(dx) \le&~ \int_{\RR} g(t,Y_{t^-},x) K(dx),
\end{align*}
for all $g(t,y,\cdot) \in \cF_{idcx}~ \text{or}~ \cF_{icx}$ such that the integrals exist, we get by Theorem \ref{thm:orderingidcxp} and \ref{thm:orderingicxp} that
\begin{align*}
E[f(Y_T)] \le E[f(X_T)]. 
\end{align*}
Note that since we assumed the semimartingales to be one-dimensional, the second inequality yields also the inequality in the positive semidefinite order.\\
As usual a reverse ordering in the differential characteristics provides the inverse inequality.
\end{example}

\begin{remark}
\begin{itemize}
\item[1.] The condition that $X$ is of type C can be modified. It is used to obtain that $\supp\left (P^{X_t}\right ) = \RR$ for all $t \in [0,T]$. In \citet[Theorem 24.10]{Sa99} there is another condition for this support, namely $0 \in \supp(\nu)$, $\supp(\nu) \cap (0,\infty) \neq \emptyset$ and $\supp(\nu) \cap (-\infty,0) \neq \emptyset$.

\item[2.] If we are interested in the integral stochastic order generated by increasing directionally convex functions, we can omit the condition on the L\'{e}vy measure \eqref{eq:existencesmoothdensity}. By \citet[Theorem 3.12.9]{MS02} this stochastic order is generated by infinitely differentiable directionally convex functions and we hence do not need a smooth density. This is a consequence of the convolution argument in the last section, equation \eqref{eq:G_gconvolutionLevy}. 
\end{itemize}
\end{remark}

Next we give an example for a comparison of a L\'{e}vy process and an It\^{o} semimartingale by Theorems \ref{thm:orderingdcxemm} and \ref{thm:orderingcxemm}.

\begin{example}[Comparison of $\RR$-valued L\'{e}vy processes and It\^{o} semimartingales]
\label{ex:exampleorderingemm}
Let $X$ be an $\RR$-valued L\'{e}vy process as in Example \ref{ex:exampleorderingp}, without assuming inequality \eqref{eq:levyprocessspecial}; we do not need that the semimartingales are special. Further, let $Y$ be an It\^{o} semimartingale:
\begin{align*}
Y_t = y_0 + \int_0^t \beta_s ds + \int_0^t \delta_s dB_s + \int_0^t \int_{|y| \le 1} y \tilde{\mu}^Y(ds,dy) +\int_0^t \int_{|y| \ge 1} y \mu^Y(ds,dy),
\end{align*}
where $y_0 \in \RR$, $\beta$ and $\delta$ are adapted processes such that the integral exist, $(B_t)_{t \in [0,T]}$ is a standard Brownian motion and $\tilde{\mu}$ is the compensated jump measure of $Y$. As before the compensator is of the form $\nu(dt,dx) = dt~ n_t(dx)$. We consider an integrable convex or directionally convex function $f$.\\
We assume the existence of e.m.m. for $X$ and $Y$. As we have seen in Example \ref{ex:exampleorderingp}, it is advantageous if $X$ is a L\'{e}vy process under $Q_1$. Therefore, we assume that $Q_1$ is structure preserving, i.e. $X$ remains a L\'{e}vy process under $Q_1$. By \citet[Theorem 4.21]{RS11} we see directly that the differential semimartingale characteristics of $X$ under $Q_1$ are given by $(0,c^2,h K)$, where $h: \RR \to \RR_+$ is a Borel measurable function such that 
\begin{align*}
\int_\RR (\sqrt{h(x)} - 1)^2 K(dx) < \infty.
\end{align*} 
This choice of $Q_1$ leads to the validity of Assumption~$(i)$. Assumption~$(ii)$ is achieved as in Example \ref{ex:exampleorderingp} since $Q_1 \sim Q_2$. Further conditions $(iii)$ and $(iv)$ are assumed to be in force.

The change of measure does not affect the integrator of a good version of the semimartingale characteristics. This follows directly form Girsanov's theorem for semimartingales, cf. \citet[Theorem III.3.24]{JS03}. Hence, the integrator remains the identity. Consequently, the differential characteristics of $Y$ alter to $(0,\delta^2,Z \eta)$, where $Z: \Omega \times \RR_+ \times \RR \to \RR$ is non-negative, $\mathscr{P} \otimes \mathscr{B}(\RR)$-measurable and fulfills the conditions of \citet[Theorem III.3.24]{JS03}. If now $dt \times Q_2$ almost surely the differential characteristics are ordered,
\begin{align*}
\delta^2_t \le&~ c^2,\\
\int_{\RR} g(t,Y_{t^-},x) Z_t(x) \eta_t(dx) \le&~ \int_{\RR} g(t,Y_{t^-},x) h(x) K(dx),
\end{align*}
for all $g(t,y,\cdot) \in \cF_{idcx}~ \text{or}~ \cF_{icx}$ such that the integrals exist, we obtain from Theorem \ref{thm:orderingdcxemm}, \ref{thm:orderingcxemm} that
\begin{align*}
E_{Q_2}[f(Y_T)] \le E_{Q_1}[f(X_T)].
\end{align*}
\end{example}

So far we only considered one-dimensional examples because we then obtain that the support is the whole space $\RR$. We now give an example in higher dimensions.

\begin{example}[Comparison of Markovian special It\^{o} semimartingales and special It\^{o} semimartingales]
\label{ex:exampleorderingitop}
Let $X$ be a $d$-dimensional Markovian special It\^{o} semimartigale. Then its differential charateristics with respect to the Lebesgue measure are deterministic functions of time and state
\begin{align*}
b^i_t =&~ b^i(t,X_{t^-}),\\
c^{ij}_t = &~ c^{ij}(t,X_{t^-}),\\
K_{\omega,t}(dx) = &~ K(t,X_{t^-}(\omega))(dx).
\end{align*}
We compare $X$ to an special It\^{o} semimartingale $Y$. Therefore, let $Y$ be as in Example \ref{ex:exampleorderingp}. Assume that $\supp\left (P^{Y_t}\right ) \subset \supp\left (P^{X_t}\right )$ for all $t \in [0,T]$. This case is considered in \citet{BR06,BR07a}. Let $f$ be an increasing convex or increasing directionally convex integrable function.

For the regularity of $G_f$ in time we assume that the transition operators form a strongly continuous evolution system, then the differentiability follows as in equation \eqref{eq:semigroupcontinuityGg}. For the regularity in the space variable, we assume that the transition probabilities are regular enough, see Section \ref{subsubsec:regularity}. Further, we assume that $X$ propagates increasing convexity or increasing directional convexity.

Recall that by the Markov property of $X$, $G_f(t,X_t)$ is a martingale. From the condition on the support of $Y$ we gain condition $(ii)$. Assumptions~$(iii)$ and $(iv)$ are imposed.

Now the $dt \times P$ almost sure ordering of the differential semimartingale characteristics for all $i,j \le d$,
\begin{align}
\begin{split}
\label{eq:exampleorderingItosemimartingalep}
\beta^i_t \le &~ b^i(t,X_{t^-}),\\
\delta^{ij2}_t \le &~ c^{ij}(t,X_{t^-}),\\
\int_{\RR^d} g(t,Y_{t^-},x) \nu^Y(dx) \le&~ \int_{\RR^d} g(t,Y_{t^-},x) K(t,X_{t^-})(dx),
\end{split}
\end{align} 
for all $g(t,y,\cdot) \in \cF_{idcx}~ \text{or}~ \cF_{icx}$ such that the integrals exist, yields by Theorem \ref{thm:orderingidcxp}, \ref{thm:orderingicxp} that 
\begin{align*}
E[f(Y_T)] \le E[f(X_T)]
\end{align*}
when $f$ is integrable and increasing directionally convex. If $f$ is integrable and increasing convex we need to exchange the second inequality in \eqref{eq:exampleorderingItosemimartingalep} to $\delta_t^2 \le_{psd} c(t,X_{t^-})$ to gain the same inequality.
\end{example}

Note that this example differs a bit from \citet{BR06,BR07a}. There the semimartingales are defined on different probability spaces. As consequence of this setting on the right-hand side of the Inequalities \eqref{eq:exampleorderingItosemimartingalep} the differential semimartingale characteristics of $X$ have to be evaluated at $Y_{t^-}$. This is not necessary in the framework here. However, we consider in the proofs $f(X_T)$ conditioned on $X_t = Y_t$ and hence in this setting we could interchange $X$ and $Y$. This then leads to inequalities as in \citet{BR06,BR07a}.

Next we give an application of Theorems \ref{thm:orderingidcxp}, \ref{thm:orderingicxp} to the case when the integrator $A$ in the good version of the semimartingale characteristics is not the identity.

\begin{example}[Comparison result for extended Grigelionis Processes] We assume the semimartingale $X$ to be an extended Grigelionis process. This sort of processes are used for example in \citet{Ka98}. A special semimartingale is called an extended Grigelionis process if there exists a discrete set $\Theta \subset \RR_+ \setminus \{0\}$ so that the increasing process of a good version of the semimartingale characteristics is given by 
\begin{align*}
A_t = t + \sum_{s \le t} \ind_\Theta(s).
\end{align*}
Intuitively this definition means that we have an It\^{o} semimartingale plus jumps at fixed times. The semimartingale characteristics $(B,C,\nu)$ of $X$ then have the form
\begin{align*}
B^i_t = &~ \int_0^t b^i_s ds + \sum_{s \in \Theta \cap [0,t]} b^i_s,\\
C^{ij}_t = &~ \int_0^t c^{ij}_s ds,\\
\nu([0,t]\times G) = &~ \int_0^t K_s(G) ds + \sum_{s \in \Theta \cap [0,t]} K_s(G)~~\text{for any } G \in \mathscr{B}(\RR^d),
\end{align*} 
where $(b_t)_{t \in [0,T]}$ is a predictable $\RR^d$-valued process, $(c_t)_{t \in [0,T]}$ is a predictable $\RR^{d \times d}$-valued process and $K$ is a transition kernel from $(\Omega \times \RR_+, \cP)$ into $(\RR^d,\mathscr{B}(\RR^d))$. Further, we assume that $X$ is a Markov process with transition probabilities that are regular enough as in Example \ref{ex:exampleorderingitop}.

We compare $X$ to another extended Grigelionis process $Y$ with characteristics $(\tilde{B},\tilde{C},\tilde{\nu})$ with respect to $\tilde{A}_t = t + \sum_{s \le t} \ind_{\tilde{\Theta}}(s)$. We assume that $\supp(P^{Y_t}) \subset \supp(P^{X_t})$ for all $t \in [0,T]$. To apply Theorems \ref{thm:orderingidcxp} and \ref{thm:orderingicxp}, we need to find a common integrator for a good version of the semimartingale characteristics. Choosing $A' = A + \tilde{A}$, cf. Remark \ref{rem:onthm:orderingdcxemm}, the differential characteristics of $X$ change to $\left (b \ind_{(\tilde{\Theta} \setminus \Theta)^c}, c, K \ind_{(\tilde{\Theta} \setminus \Theta)^c}\right )$, the differential characteristics of $Y$ change accordingly. Let $f$ be an increasing convex or increasing directionally convex integrable function.

The differentiability of $G_f$ follows as in Example \ref{ex:exampleorderingitop}. Further we assume the propagation of order. Hence, we have condition $(i)$. Condition $(ii)$ follows from the Markov property and the assumptions on the supports; conditions $(iii)$ and $(iv)$ are imposed.

Then an ordering of the semimartingale characteristics yields an ordering of expectations. For more details see \citet{Ko19}.
\end{example}

In Corollary \ref{cor:girsanovemm} we mentioned already the particular simplification if we consider a semimartingale under two different e.m.m. We apply this simplification in the case of a L\'{e}vy process.

\begin{example}[Comparison of e.m.m. for a one-dimensional L\'{e}vy process]
Let $X$ be a one-dimensional type C L\'{e}vy process and $Q_1$ and $Q_2$ e.m.m. of $X$. We assume that $X$ possesses a smooth Lebesgue-density under $Q_1$, for example by imposing inequality \eqref{eq:existencesmoothdensity}. Further, we assume that $Q_2$ is structure preserving (see Example \ref{ex:exampleorderingemm}). No further restrictions are put on $Q_1$. Let $f$ be an integrable function not necessarily convex or directional convex, see Remark \ref{rem:oncoremm}.

As seen in Example \ref{ex:exampleorderingemm} in this case condition $(i)$ and $(ii)$ of Corollary \ref{cor:girsanovemm} hold. Condition $(iii)$ and $(iv)$ are assumed. If we assume that the L\'{e}vy measures are $dt \times Q_1$ almost surely ordered,
\begin{align*}
\int_{\RR} H_{G_f}(t,Y_{t^-},x) K^1_t(dx) \le&~ \int_{\RR} H_{G_f}(t,Y_{t^-},x) K^2(dx),
\end{align*}
then we obtain from Corollary \ref{cor:girsanovemm}
\begin{align*}
E_{Q_1}[f(X_T)] \le E_{Q_2}[f(X_T)].
\end{align*}
\end{example}

\bibliographystyle{chicago}

\end{document}